\documentclass[12pt]{amsart}
\usepackage[english]{babel}
\usepackage[utf8]{inputenc}
\usepackage[T1]{fontenc}
\usepackage{amsmath, amsfonts, amsthm, amssymb, mathrsfs}
\usepackage[all,cmtip]{xy}
\usepackage{graphicx}
\usepackage{color}
\usepackage[hidelinks]{hyperref}
\usepackage{tikz}
\usepackage{algorithm2e}

\usepackage{mathtools}
\usepackage{microtype}

\newtheorem{theorem}{Theorem}[section]
\newtheorem{lemma}[theorem]{Lemma}
\newtheorem{proposition}[theorem]{Proposition}
\newtheorem{corollary}[theorem]{Corollary}

\theoremstyle{definition}

\newtheorem{remark}[theorem]{Remark}
\newtheorem{example}[theorem]{Example}

\newcommand{\NN}{\mathbb{N}}
\newcommand{\ZZ}{\mathbb{Z}}
\newcommand{\RR}{\mathbb{R}}

\newcommand{\ct}{\operatorname{ct}}
\newcommand{\KW}{\operatorname{KW}}
\newcommand{\HMT}{\operatorname{HMT}}
\newcommand{\SL}{\mathrm{SL}}
\newcommand{\Pure}[1]{\mathrm{Pure}_{\pi_1\neq 1}^{#1}}
\newcommand{\zz}[1]{}

\begin{document}

\title[Fundamental groups of small simplicial complexes]{Fundamental groups of small simplicial complexes}
\date{}

\author[D. Govc]{Dejan Govc$^{*,**}$}
\author[W. Marzantowicz]{Wac{\l}aw Marzantowicz$^{***}$}
\author[{\L}. P. Michalak]{\\{\L}ukasz Patryk Michalak$^{**,***}$}
\author[P. Pave{\v s}i{\'c}]{Petar Pave{\v s}i{\'c}$^{*,**}$}

\thanks{$^{*,**}$ Supported by the Slovenian Research Agency program P1-0292 and grant J1-4001.}
\thanks{$^{***}$ Supported by the Polish Committee of Scientific Research Grant
Sheng 1 UMO-2018/30/Q/ST1/00228.}

\address{$^{*}$ \;Faculty of Mathematics and Physics, University of Ljubljana,  Jadranska 19, 1000 Ljubljana, Slovenija}
\email{dejan.govc@gmail.com, ORCID: \href{https://orcid.org/0000-0002-7156-7966}{0000-0002-7156-7966}}
\email{petar.pavesic@fmf.uni-lj.si, ORCID: \href{https://orcid.org/0000-0001-7827-9900}{0000-0001-7827-9900}}

\address{$^{**}$ \;Institute of Mathematics, Physics and Mechanics, Jadranska 19, 1000 Ljubljana, Slovenija}

\address{$^{***}$ \;Faculty of Mathematics and Computer Science, Adam Mickiewicz University,	Pozna{\'n}, ul. Uniwersytetu Pozna{\'n}skiego 4, 61-614 Pozna{\'n}, Poland}

\email{marzan@amu.edu.pl, ORCID: \href{https://orcid.org/0000-0001-5933-9955}{0000-0001-5933-9955}}
\email{lukasz.michalak@amu.edu.pl, ORCID: \href{https://orcid.org/0009-0002-4821-3809}{0009-0002-4821-3809}}


\zz{55U10 - Simplicial sets and complexes in algebraic topology
20F65 - Geometric group theory
57Q15 - Triangulating manifolds
57-04 - Software, source code, etc. for problems pertaining to manifolds and cell complexes
55.40 - Homotopy theory, general

05E45 - Combinatorial aspects of simplicial complexes
20F05 - Generators, relations, and presentations of groups
20F34 - Fundamental groups and their automorphisms (group-theoretic aspects)
57M05 - Fundamental group, presentations, free differential calculus}
\maketitle

\centerline{\emph{Dedicated to the memory of Frank Hagen Lutz}}

\begin{abstract}
The number of nonisomorphic simplicial complexes with up to $n$ vertices increases super-exponentially with $n$,
which makes exhaustive computation of invariants associated with such complexes a daunting task. In this paper we
provide a complete list of groups that arise as fundamental groups of simplicial complexes with at most $8$ vertices. In addition we give many examples
of fundamental groups of complexes with $9$ vertices although the complete classification seems to be beyond
reach at the moment.

Our results lead to many applications, including progress on the Bj\"orner--Lutz\linebreak
conjecture regarding vertex-minimal triangulations of the Poincar\'e homology sphere,\linebreak improved recognition criteria for
PL triangulations of manifolds and computation of the Karoubi--Weibel invariant for many groups.
\end{abstract}

\section{Introduction}\label{sec:Introduction}

Fundamental group is the most important homotopy invariant of low
dimensional spaces and spaces that admit small triangulations
as it gives a lot of information about their structure and often completely determines their homotopy type.
In this paper we consider the problem of which groups arise as fundamental groups of simplicial
complexes on a given number of vertices. For simplicial complexes
with at most $8$ vertices we obtain a complete list of their fundamental groups:\\[2mm]
{\bf Theorem \ref{thm:n=8}} \emph{
An abstract group can be realized as the fundamental group of a simplicial complex $K$ on at most $8$ vertices if, and only if, it
is of the form $\pi_1(K)\cong G\ast F_r,$ where $G$ is one of the following groups:
\begin{itemize}
	\item $\ZZ_2$, $\ZZ_3$, $\ZZ_4$,
	\item $\ZZ\times\ZZ$, $\ZZ\rtimes\ZZ$ (fundamental groups of the torus and the Klein bottle, respectively),
	\item $B_3$ (the braid group on $3$ strands or
	the trefoil knot group),
\end{itemize}
and $F_r$ is a free group of rank 
$r \leq 21$ where the range of admissible ranks $r$ depends on $G$ and admits the value $0$ in every case.
}

\

Passing to the next case, there are about $7.89\times 10^{35}$ isomorphism classes of simplicial
complexes and about $5.33\times 10^{19}$
isomorphism classes of $2$-pure simplicial complexes on $9$ vertices. Even after several geometric
reductions of the number of cases that
need to be considered we were not able to check all of them. Nevertheless, our computations gave
many new fundamental groups
that are listed below. For clarity we omit the free factors that arise from the addition
of edges to a given triangulation, which can be done if the $1$-skeleton is not yet a complete graph.
We also omit the previously listed fundamental groups of complexes with $8$ vertices.\\[2mm]
{\bf Theorem \ref{result:9_vertices}} \emph{The following groups can be realized as fundamental groups of simplicial complexes on $9$ vertices:
\begin{itemize}
	\item $\ZZ_m$ for $m\in \{5,6,7,8,9\}$,
	\item $\ZZ_2 \ast \ZZ_2 = D_\infty$ (infinite dihedral group),
	\item $\ZZ_2 \ast \ZZ_3$,
	\item $\ZZ \times \ZZ_m$ for $m\in \{2,3\}$,
	\item $D_6$ (dihedral group of order $6$),
	\item $Q_8$ (quaternion group of order $8$),
	\item $BS(2,1)$, $BS(2,2)$, $BS(2,-2)$, $BS(3,1)$, $BS(3,-1)$ (Baumslag--Solitar groups),
	\item $\pi_1(M_2)$ (the fundamental group of a closed orientable surface of genus $2$),
	\item $\pi_1(N_g)$ for $g\in \{3,4,5\}$ (the fundamental group of a closed non-orientable surface of genus $g$),
	\item $\pi_1(X_{2,4}) = \langle\, a,\, b\, |\, a^4 = b^2\, \rangle$ (see \cite[Example 1.24]{Hatcher}).
\end{itemize}}

There are many problems that can be reduced to the size of the
simplicial complex that realizes a given fundamental group. For
example, the Bj\"orner--Lutz conjecture \cite{Bjorner-Lutz}
states that every triangulation of the Poincar\'e homology sphere
requires at least $16$ vertices. That value comes from a specific
triangulation obtained through a computer search.
B.~Bagchi and B. Datta \cite{Bagchi-Datta} proved that every such triangulation requires at least $12$ vertices. Based on our computations
we are able to obtain the following improvement:\\[2mm]
{\bf Theorem \ref{thm:Poincare}} \emph{
Every triangulation of the Poincar\'e homology sphere
requires at least $13$ vertices (and at least $14$, for triangulations that admit at least one bistellar flip).}

\

As an immediate consequence we obtain a useful recognition
criterion for PL triangulations that, roughly speaking, states that every triangulation of a $4$-dimensional manifold
in which the link of each vertex has at most $12$ vertices, is
a PL triangulation.
Another application of our result is the computation of the
Karoubi--Weibel invariant of many groups. We will give more
details below, after the introduction of suitable definitions.

\subsection{Prior work.}

Fundamental group is particularly important for the study of low-dimensional spaces and spaces with a small number of cells. It also represents the main bridge between topology and group theory, which leads to the study
of topological invariants associated to groups.
Here we will give a brief overview of the results that are
closely related to our work.

Our initial motivation for this work was a conjecture
regarding the minimal triangulation of the Poincar\'e homology sphere.
A. Bj\"orner and F. H. Lutz \cite{Bjorner-Lutz} found a triangulation of the Poincar{\'e} sphere with $16$ vertices. They started from known $18$ and $24$-vertex triangulations and used their bistellar flip program which randomly (with some priority) performs bistellar moves in order
to reduce the size of the triangulation. As extensive search failed to
find further reductions, they conjectured \cite[Conjecture 6]{Bjorner-Lutz} that $16$ is in fact the minimum number of vertices needed to triangulate the Poincar{\'e} sphere.

It follows from the work of D. W. Walkup \cite{Walkup} that every triangulation of the Poincar\'e sphere requires at least $11$ vertices,
so the gap between the lower and upper bound is quite large. In \cite{Bagchi-Datta}
B. Bagchi and B. Datta  studied the collapsibility
of small acyclic complexes and deduced that at least $12$ vertices
are needed for a triangulation. In \cite{Pavesic} the fourth author studied
minimal triangulations of manifolds whose fundamental group is not
free. His results indicated that a more detailed study of possible
triangulations of the binary icosahedral group (the fundamental group
of the Poincar\'e sphere) could lead to some progress on the Bj\"orner--Lutz conjecture.

A different line of studies that led to this work was initiated
by M. Karoubi and C.~Weibel \cite{Karoubi-Weibel} who studied
good covers of spaces and introduced the \emph{covering type} $\ct(X)$
of a space $X$ as the minimal number of elements in a good open
cover of a space homotopy equivalent to $X$. Covering type is
a homotopy invariant of the space and is related to minimal triangulations
via the Nerve Theorem. In fact by \cite[Theorem 1.2]{GMP}, $\ct(X)$ equals the number of
vertices in a minimal homotopy triangulation of $X$. I. Babenko, F. Balacheff and G. Bulteau
\cite{BabenkoBB} related minimal triangulations to group theory by defining the
\emph{simplicial complexity} of a group $G$ as the minimal number
of facets in a $2$-dimensional simplicial complex whose fundamental group is isomorphic to $G$. They used
the simplicial complexity of $G$ as a discrete estimate of Gromov's systolic area of $G$ (see
\cite{BabenkoBB} for further details). Simplicial complexity of $G$ is directly related to the minimal
number of vertices in a $2$-dimensional
simplicial complex $K$ whose fundamental group is $\pi_1(K)\cong G$,
which led I. Babenko and T. Moulla \cite{Babenko-M} to introduce the
\emph{Karoubi--Weibel invariant} $\KW(G)$ as
\[
KW(G) = {\underset{\pi_1(X)=G}{\,\min\,}} \ct(X).
\]
In this paper, we will, among other things, greatly extend the computations of $\KW(G)$ given in \cite{Babenko-M}.
At the moment, most known results about $\KW(G)$ describe its asymptotic behavior for groups in a particular family. While the upper bound can be  obtained
constructively through a specific realization of the complex with a given group, it is usually much harder to find meaningful lower bounds. One such instance appears
in a recent paper of F. Frick and M. Superdock \cite{Frick-Superdock}, who give an asymptotic lower bound $k^{3/4}$ for $\KW(\ZZ^k)$. In addition, they describe
a family of complexes whose fundamental group is $\ZZ^k$ with $4k+(-1)^{k+1}$ vertices, so the growth of $\KW(\ZZ^k)$ is at most linear in $k$.

There are several related results on homology groups. A.~Newman \cite{Newman_torsion} carried out asymptotic analysis of the torsion in homology groups.
For example, for a given finite abelian group $G$ and $d\geq 2$ Newman estimated the minimum number of vertices $T_d(G)$ of a simplicial complex such that
the torsion part of its $(d-1)$st integral homology group is isomorphic to $G$. He showed that $T_d(G)$ is asymptotically $(\log|G|)^{1/d}$ which for $d=2$ yields a lower asymptotic estimate for $\KW(G)$. Further asymptotic results were obtained by D. Lofano and F.\,H. Lutz \cite{Lofano-Lutz}, who gave a construction, using Hadamard matrices, of a family of complexes $\HMT(n)$ with $5n-1$ vertices such that $|H_1(\HMT(n))|=n^{\frac n2}$.

\subsection{Our contributions.}

We present a theoretical and computational approach that allows us to determine all isomorphism types
of fundamental groups of simplicial complexes on up to $8$ vertices, up to a free group factor.
As mentioned, the growth
of the number of nonisomorphic simplicial complexes on a given
number of vertices is so fast that it makes direct computation of
all possible cases unfeasible. This problem may be somewhat mitigated
by a suitable reduction of the number of cases that need to be considered.
Thus our geometric reductions given in sections \ref{subsec:scx up to 7} and \ref{subsec:General structural properties} 
are of independent interest and could be applied in other similar situations.
While the number of nonisomorphic simplicial complexes on $7$ vertices
is already quite big (around $490$ million) our geometric simplifications
allow us to examine all possible cases purely by theoretical means.
The case of $8$ vertices presents the first significant difficulty due to the number of
(about $1.4\times 10^{18}$ nonisomorphic) complexes and group presentations.
By implementing our algorithm and running a computer program, we determined a complete list of fundamental groups
of these complexes as given in Theorem \ref{thm:n=8}.
For $9$ vertices the number of nonisomorphic simplicial complexes is
about $7.89\times 10^{35}$, so our geometric reductions are
unfortunately not sufficient to examine all the relevant cases. Despite this, we managed to identify many
new fundamental groups. These are listed in Theorem \ref{result:9_vertices}.

The introduction of the Karoubi--Weibel invariant in \cite{Babenko-M}
was motivated by its relation to Gromov's systolic area of a group.
However, apart from some elementary cases, there are very few instances of groups $G$ for which the precise value of $\KW(G)$
is known. Among the known examples, we can mention the nonabelian free groups (fundamental groups of $1$-dimensional complexes),
the cyclic group $\ZZ_2$ with $\KW(\ZZ_2)=6$ (minimal triangulation of the projective plane) and the free abelian group
on two generators with $\KW(\ZZ^2)=7$ (minimal triangulation of the torus). However, vertex-minimality of triangulations is a delicate issue.
For example, it is not immediately clear that the vertex-minimal triangulations
of surfaces are also the vertex-minimal $2$-complexes for the
corresponding fundamental groups. We will tackle this problem in
a forthcoming paper.  In this paper we
completely determine all groups whose Karoubi--Weibel invariant
is less than or equal to $8$ and provide many examples of groups whose
Karoubi--Weibel invariant is $9$.

Among the specific consequences of our computations is that for
the binary icosahedral group $\SL(2,5)$ we have $\KW(\SL(2,5))\ge 9$.
Combined with some geometric considerations this implies that every
triangulation of the Poincar\'e sphere requires at least $13$ vertices.
If~it turns out, as we expect, that $\KW(\SL(2,5))=10$, it would
improve the lower bound for such a triangulation to $14$ (or $15$ for
triangulations that admit a bistellar flip). Furthermore, our results
also show that $\KW(G)\ge 9$ for any nontrivial perfect group $G$, which
implies that a combinatorial triangulation of a $d$-dimensional homology
sphere (over $\ZZ$) requires at least $d+10$ vertices. This improves the
result of Bagchi and Datta \cite{Bagchi-Datta} by $1$.

\subsection{Organization of the paper.}
In Section \ref{sec:preliminary} we review the results on the number of simplicial complexes on a given number of vertices,
according to different types or relations: all complexes (Dedekind numbers), their isomorphism classes
(reduced Dedekind numbers), $2$-pure complexes (the number of $3$-uniform hypergraphs) and homotopy
types (A. Newman's result). In Section \ref{sec:up_to_7_ver} we discuss the fundamental groups of complexes
with up to $7$ vertices. These can be determined either by direct geometric arguments or by using standard
software. Section \ref{sec:8_ver} contains the main result
of the paper, namely the classification of fundamental groups of complexes on $8$ vertices. We establish a structural result
for complexes (Theorem \ref{theorem:structural_theorem}) that allows us to obtain every occurring group up to a free group factor,
and we discuss the algorithm and its implementation. We also present the consequences of these results for the vertex-minimal triangulations of the
Poincar{\'e} sphere and the covering type of the complement of trefoil knot. In Section \ref{sec:applications} we discuss some
computational issues that arise in the analysis of complexes on $9$ vertices and present
many examples of groups whose Karoubi--Weibel invariant is equal to $9$. We conclude with some further applications related
to aspherical spaces and to recognition of combinatorial triangulations.

\subsection{Acknowledgements.}
The collaboration on this project was initiated during the workshop \emph{Some problems
in applied and computational topology} in March 2023 in B\k{e}dlewo  with the kind support
of the Banach Center. Sadly, this turned out to be the last time that we had the opportunity
to enjoy the company and discuss mathematics with Frank Lutz,
who unexpectedly passed away on 10th of  November 2023.

\section{Numbers of complexes of different types}\label{sec:preliminary}

\subsection{Basic facts and notions}
Let $K$ be a simplicial complex. We say that:
\begin{itemize}
	\item a simplex of $K$ is \emph{maximal} if it is not a proper face of another simplex of $K$.
	\item a simplex of $K$ is \emph{free} if it is a proper face of exactly one simplex of $K$.
	\item $K$ is \emph{$n$-pure} if every simplex in $K$ is a face of an $n$-dimensional simplex of $K$.
\end{itemize}
We will always assume that $K$ is non-empty and is realized in some Euclidean space.

By the Grushko decomposition theorem \cite[Section 4.1]{Magnus-book}
every finitely generated group $G$ can be decomposed as a free product
\[
G = G_1 \ast G_2 \ast \ldots G_m \ast F_r,
\]
where each $G_i$ is a nontrivial, freely indecomposable group, and $F_r$ is a free group of rank
$r\geq 0$ ($F_0 = 1$ is the trivial group and $F_1 = \ZZ$). The decomposition is unique up to
isomorphism and permutation of groups $G_i$. Observe that the addition of an edge to $K$
(which is possible if the 1-skeleton of $K$ is not a complete graph)
increases the free rank of its fundamental group by $1$,
so we will be mostly interested in the non-free part of
$\pi_1(K)$.

\subsection{Dedekind numbers}\label{subsec:Dedekind numbers}
Let us begin by reviewing the numbers of simplicial complexes that
need to be examined.

The $n$th \emph{Dedekind number} $d_n$ is the number of all monotone Boolean functions on $n$ variables. It has several equivalent reformulations (cf. \cite{Jakel,Kisielewicz,Hirtum}), one of them is that it equals the number of all abstract simplicial complexes on $n$ vertices, plus $1$. Since the empty set is also counted as an abstract simplicial complex, it follows that $d_n-2$ is the number of all (geometric) simplicial complexes on $n$ vertices.

The exact values of $d_n$ are known only for $n=0,1\ldots,9$ (the sequence \href{https://oeis.org/A000372}{A000372} in the The On-Line Encyclopedia of Integer Sequences (OEIS) \cite{OEIS}):
\begin{align*}
	2, 3, 6, 20, 168, 7581, 7828354, 2414682040998, 56130437228687557907788, \\ 286386577668298411128469151667598498812366.
\end{align*}
The last number in the series, $d_9 \approx 2.86 \times 10^{41}$ was determined only recently with the help of a supercomputer by C.~J{\"a}kel~\cite{Jakel} and L. Van Hirtum et al. \cite{Hirtum}.

It is interesting to note that in 1988  A. Kisielewicz \cite{Kisielewicz} found an explicit summation formula for $d_n$:
  \[d_n=\sum_{k=1}^{2^{2^n}}\prod_{j=1}^{2^n-1}\prod_{i=0}^{j-1}\Big(1- b_i^kb^k_j\prod_{m=0}^{\log_2i}(1-b_m^i+b_m^ib_m^j)\Big),\] where $b_i^k=\lfloor k/2^i\rfloor-2\lfloor k/2^{i+1}\rfloor$,  and
  $\lfloor - \rfloor$  is the floor function.
 However, the summation is performed over too many terms to be useful if $n$ is large (even if $n=8$).

\subsection{Number of isomorphism classes of simplicial complexes}

A variant of the above is the $n$th {\em reduced Dedekind number} $r_n$, which is one more than the number of isomorphism classes of abstract simplicial complexes on $n$ vertices.

The values of $r_n$ are also known only for $n=0,1,\ldots,9$ (the sequence \href{https://oeis.org/A003182}{A003182} in the OEIS \cite{OEIS}):
\begin{align*}
	2, 3, 5, 10, 30, 210, 16353, 490013148, 1392195548889993358, \\ 789204635842035040527740846300252680.
\end{align*}
The last of these $r_9\approx 7.89\times 10^{35}$ was recently computed by B. Pawelski \cite{Pawelski}.

\subsection{Number of isomorphism classes of uniform hypergraphs}
In order to compute fundamental groups up to a free factor, it is enough
to consider only $2$-pure simplicial complexes in which all maximal faces are triangles (cf. Lemma \ref{lemma:2-pure}). This combinatorial description can be expressed in terms of hypergraphs: a $2$-pure complex corresponds uniquely to a hypergraph on $n$ vertices with each hyperedge of order $3$. Such hypergraphs are called \emph{$3$-uniform}.

The problem of enumeration of isomorphism classes of uniform hypergraphs is classical --- see the book \cite{Harary-Palmer} of F. Harary and E.\,M. Palmer or the paper \cite{Qian} of J.~Qian providing an explicit formula by using P{\'o}lya’s counting theory and Burnside’s counting lemma. The numbers $h_3(n)$ of isomorphism classes of $3$-uniform hypergraphs on $n$ vertices form the sequence \href{https://oeis.org/A000665}{A000665} in the OEIS \cite{OEIS}. For $n=0,1,\ldots,9$ they are as follows:
\[
1, 1, 1, 2, 5, 34, 2136, 7013320, 1788782616656, 53304527811667897248.
\]
Note that these classes include the empty hypergraph.

Let $\mathcal{P}(n)$ denote the set of all partitions of a natural number $n$. A partition $P\in \mathcal{P}(n)$ can be written in two forms $P: p_1 + p_2 + \cdots + p_q = n = 1\alpha_1 + 2\alpha_2 + \cdots + n\alpha_n$, where $\alpha_i$ is the number of integers $i$ in the partition $P$.

\begin{theorem}[{\cite[Theorem 2.4 and Corollary 2.6]{Qian}}]
	The number of isomorphism classes of $3$-uniform hypergraphs on $n$ vertices is equal to
	\[
	h_3(n) = \sum_{P \in \mathcal{P}(n) } \frac{1}{1^{\alpha_1}2^{\alpha_2}\cdots n^{\alpha_n}\alpha_1!\alpha_2!\cdots\alpha_n!} 2^{\tau_3(P)},
	\]
	where
	\begin{align*}
	\tau_3(P) &= \sum_{i=1}^q \left\lceil{\frac{(p_i-1)(p_i-2)}{6}}\right\rceil 
	+ \sum_{i<j<h \in \{1,\ldots,q\}} \frac{p_i p_j p_h}{{\rm lcm}(p_i,p_j,p_h)}
	\\&+
	\frac{1}{2}\sum_{i<j} \gcd(p_i,p_j)\left(p_i+p_j-2+\frac{1}{2} \left|(-1)^{{\rm lcm}(p_i,p_j)/p_i} - (-1)^{{\rm lcm}(p_i,p_j)/p_j} \right| \right).
	\end{align*}
\end{theorem}

\begin{remark}
	
	Note that the expression in the second summation in the formula for $\tau_3(P)$ has been corrected in accordance with the proof of \cite[Corollary 2.6]{Qian}, Case 1, which was also pointed out by A. Howroyd in a comment on \href{https://oeis.org/A000665}{A000665} in OEIS \cite{OEIS}.
\end{remark}

We may summarize the above considerations in the following table:

	\begin{table}[h]
	\begin{tabular}{ c|ccc }
	\hline
	$n$ & $d_n-2$ & $r_n-2$ & $h_3(n)-1$ \\
	    & (all complexes) & (isomorphism classes) & ($2$-pure complexes)\\
	\hline
	1 & 1 & 1  & 0 \\
	2 & 4 & 3 & 0 \\
	3 & 18  & 8 & 1 \\
	4 & 166 & 28 & 4 \\
	5 & 7579 & 208 & 33 \\
	6 & 7828352 & 16351 & 2135 \\
	7 & 2414682040996 & 490013146 & 7013319 \\
	8 & $\sim \!5.61 \times 10^{22}$ & $\sim \!1.39 \times 10^{18}$ & $\sim \!1.79 \times 10^{12}$ \\
	9 & $\sim \!2.86 \times 10^{41}$ & $\sim \!7.89 \times 10^{35}$ & $\sim \!5.33 \times 10^{19}$ \\
	\hline
\end{tabular}

\label{table:numbers_of_complexes}
\vskip 15pt
\caption{The numbers of different types  of geometric simplicial complexes  on $n$ vertices.}
\end{table}


Let us conclude this section by mentioning an interesting estimate of
the number of homotopy types of complexes on a given number of vertices.
A. Newman \cite{Newman_homotopy} showed that the number $ht(n)$
of homotopy types of simplicial complexes on $n$ vertices is asymptotically doubly exponential in $n$ in the sense that for sufficiently large $n$ we have the estimate
\[
2^{2^{0.02n}}  \leq ht(n)  \leq d_n \leq 2^{2^n}.
\]

\section{Complexes on at most $7$ vertices}
\label{sec:up_to_7_ver}

In this section we determine the groups that arise as fundamental groups of complexes on
at most $7$ vertices. While the results are as one might expect, it is still valuable to present a unified
treatment of all cases. We occasionally rely on the following criterion to recognize wedges of spheres:

\begin{theorem}[C. T. C. Wall \cite{Wall}, Proposition 3.3]\label{wall}
	Suppose $X$ is a finite connected $2$-dimensional CW complex whose fundamental group is free. Then $X$ is homotopy equivalent to a wedge of spheres.
\end{theorem}

Suppose $X$ and $Y$ are simplicial complexes, viewed as CW complexes. Then, their product $X\times_{\mathrm{CW}}Y$ is a CW complex whose cells are of the form $\sigma\times\tau$, where $\sigma$ and $\tau$ are cells (simplices) of $X$ and $Y$, respectively. (Complexes whose cells are products of simplices are called {\em prodsimplicial complexes}.)

Recall the notion of the {\em mapping cone} $C_f$ of a map $f\colon A\to X$
(see A. Hatcher \cite{Hatcher}, Chapter 0). When $i\colon A\to X$ is the inclusion of a subspace, we
also write this as $C_i=X\cup CA$ and say that we have {\em attached the cone $CA$
along $A$ in $X$}. Since the effect of this operation is to make the inclusion of $A$ null-homotopic
in $C_i$, we also say that we have {\em coned off $A$ in $X$}. Note that if $A$ is a subcomplex of
a simplicial complex $K$ on $n$ vertices, then $K \cup CA$ has a natural structure of a simplicial
complex on $n+1$ vertices.

\newcounter{case}[theorem]
\renewcommand{\thecase}{\thetheorem(\alph{case})}

\begin{proposition}\label{split-suspension_and_cones}
	Suppose $K$ is a simplicial complex with vertex set $\{1,\ldots,n\}$. Let $A$ be the subcomplex spanned by the vertices $\{1,\ldots,m\}$ and let $B$ be the subcomplex spanned by the vertices $\{m+1,\ldots,n\}$. In the following, both of these are viewed as CW complexes.
	\begin{itemize}
		\item[(a)] \leavevmode\refstepcounter{case}\label{ssc-split}%
		If $A$ and $B$ are contractible, then  $|K|$ is homotopy equivalent to a suspension
		\[
		|K|\simeq\Sigma Z,
		\]
		where $Z$ is a subcomplex of $A\times_{\mathrm{CW}}B$. The dimension of $Z$ is equal to
		\[
		\dim Z=\max\{\dim\sigma\mid\sigma\in K\;\text{has vertices both in $A$ and $B$}\}-1.
		\]
		\item[(b)]\leavevmode\refstepcounter{case}\label{ssc-cone}%
		If $B$ is contractible, then $|K|$ is homotopy equivalent to a mapping cone
		\[
		|K| \simeq C_f,
		\]
		where $f\colon Z\to A$ is a map defined on a certain subcomplex $Z\leq A\times_{\mathrm{CW}}B$. In~particular, if $A$ is connected, then for some normal subgroup $H \unlhd \pi_1(A)$ we get
		\[
		\pi_1(K) = (\pi_1(A)/H) \ast F_{r-1}.
		\]
	\end{itemize}	
\end{proposition}

\begin{proof}
	Recall that $|K|$ can be understood as a subspace of the join $A\ast B=(A\times B\times I)/_\sim$, where $A\times B\times\{0\}$ is collapsed to $A$ and $A\times B\times\{1\}$ is collapsed to $B$ (see Hatcher \cite{Hatcher}, Chapter 0). We have a well-defined projection $q\colon A\ast B\to I$ given by $q([a,b,t])=t$. Let $U=|K|\cap q^{-1}[0,\frac12]$ and $V=|K|\cap q^{-1}[\frac12,1]$. Write $Z:=U\cap V=|K|\cap q^{-1}\{\frac12\}$. Note that $U$ deformation retracts to $q^{-1}\{0\}=A$ and $V$ deformation retracts to $q^{-1}\{1\}=B$.
	
	If $A$ and $B$ are both contractible, we have $|K|\simeq\Sigma Z$ and $Z$ is a subspace of $A\times B\times\{\frac12\}$. To determine the dimension of $Z$, simply note that $Z$ intersects each simplex that has vertices in both $A$ and $B$ along a subset of codimension~$1$.
	
	For the second statement, since $B$ is contractible, $V\simeq V/q^{-1}\{1\}\approx CZ$. On the other hand, $A\simeq M_f$, the mapping cylinder of the map $f\colon Z\to A$ defined as $f([a,b,\frac12])=a$. This decomposition shows that $|K|/B\simeq C_f$, as required.
	
	If $Z$ has $r$ connected components, write it as $Z=Z_1\cup\ldots\cup Z_r$. We can then further decompose the mapping cone as follows:
	\[
	C_f=M_f\cup CZ\simeq (M_f\cup CZ_1\cup\ldots\cup CZ_r)/_\sim\simeq (M_f\cup CZ_1\cup\ldots\cup CZ_r)\vee\bigvee^{r-1} S^1,
	\]
	where each cone $CZ_i$ has a separate apex, but then we identify them using $\sim$. Note that the second homotopy equivalence uses connectedness of $A$. Therefore $C_f$ is built from $A\simeq M_f$ by attaching $r$ cones, which has the effect of killing a subgroup of $\pi_1(A)$, and wedging with $r-1$ circles, which adds the free factor.
\end{proof}	

\begin{remark}\label{removing-a-ball}
	In the setting of Proposition \ref{ssc-cone}, if $K$ is a closed combinatorial $d$-manifold (i.e. the vertex links are combinatorial spheres), and $B$ is a combinatorial $d$-ball, we know that $K\setminus\operatorname{int}B$ has a PL collar $C$ (see \cite[Corollary 1.23]{Hudson}). Assuming $d\geq 3$, $C\simeq S^{d-1}$ is simply connected, so Seifert--van Kampen leads to a stronger conclusion:
	\[
	\pi_1(K)\cong\pi_1(K\setminus B)\ast_{\pi_1(C)}\pi_1(B\cup C)\cong\pi_1(K\setminus B)\cong\pi_1(A).
	\]
	Note that all triangulations are combinatorial if $d\leq 4$ (see \cite[Proposition 3.11]{Friedl-etAl}).
\end{remark}

\subsection{Simplicial complexes on at most $5$ vertices}

Simplicial complexes on at most $5$ vertices are easy to analyze directly so we only state the results.

It is sufficient to consider connected complexes. The first non-contractible case is the circle that appears as
the boundary of the $2$-simplex spanned by $3$ vertices. On $4$ vertices
one has a $2$-dimensional sphere or a wedge of at most $3$ circles.
On $5$ vertices there is a $3$-dimensional sphere and wedges of up to four $2$-dimensional spheres and up to six circles for the complete graph on $5$ vertices. In all the cases
the corresponding fundamental group is free of rank less than or equal to $6$. This can be summarized as follows:

\begin{proposition}\label{5vertex}
	Let $K$ be a connected simplicial complex on $n\leq5$ vertices. Then $K$ is homotopy equivalent to a wedge of spheres.
\end{proposition}

%
%
%

\subsection{Simplicial complexes on at most $6$ vertices}

On $6$ vertices there are $7828352$ simplicial complexes and $16351$ isomorphism classes so the analysis becomes a bit more involved.
Fundamental groups of such complexes can be free of rank at most $10$,
realized by subcomplexes of the $1$-skeleton of the $5$-simplex.
There is also a well-known minimal triangulation of the projective
plane, whose fundamental group is $\ZZ_2$. To see that these exhaust all possibilities, we may argue as follows.

	\begin{proposition}\label{prism}
	Let $X_1=\Delta^3$ and $X_2=\Delta^1$, viewed as CW complexes with cells given by simplices. Suppose $X\leq X_1\times_{\mathrm{CW}}X_2$ is a connected subcomplex and $\dim X\leq 2$. Then $X$ is homotopy equivalent to a wedge of spheres.
\end{proposition}

\begin{figure}[h]
	\includegraphics[width=180pt]{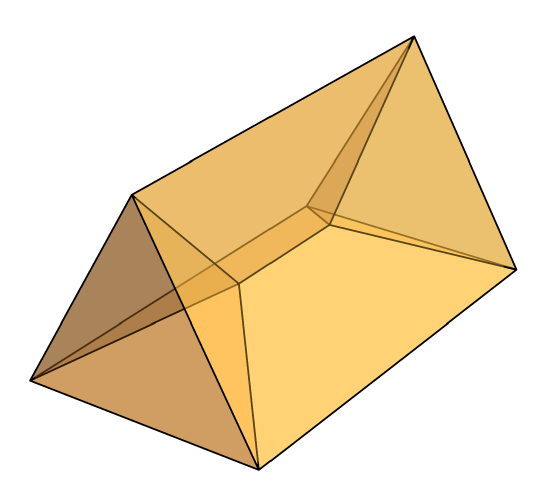}
	\caption{The complex $Y=(X_1\times_{\mathrm{CW}}X_2)^{(2)}$ can be pictured as a triangular prism in $\RR^3$, subdivided into five rooms. Two of these rooms are shaped like tetrahedra, while the other three are shaped like triangular prisms.}
	\label{figure:prism}
\end{figure}

\begin{proof}
	To prove this, we are going to use Case 3 of Theorem \ref{6vertex}. This is not circular, because only Case 2 of Theorem \ref{6vertex} depends on Proposition \ref{prism}. Note that $Y=(X_1\times_{\mathrm{CW}}X_2)^{(2)}$ has $8$ vertices, $16$ edges and $14$ two-cells (see Figure \ref{figure:prism}). Writing the vertices of $X_1$ as $\{1,2,3,4\}$ and the vertices of $X_2$ as $\{1,2\}$, the complex $X$ can therefore have two different types of two-cells:
	\begin{itemize}
		\item triangles $ijk\times l$, where $1\leq i<j<k\leq 4$ and $l\in\{1,2\}$ and
		\item squares $ij\times 12$, where $1\leq i<j\leq 4$.
	\end{itemize}
	If $X$ has no squares, it consists of two $4$-vertex simplicial complexes, connected by a collection of edges, so it is a wedge of spheres.
	
	Suppose $X$ has a square $ij\times12$. We collapse this square to an edge we call $w_iw_j$ by projecting along the edge $12$. In other words, we remove the square $ij\times12$, we collapse the edge $i\times 12$ to a new vertex $w_i$, we collapse the edge $j\times 12$ to a new vertex $w_j$ and we glue the edges $ij\times 1$ and $ij\times 2$ together to obtain a new edge $w_iw_j$. Since this procedure contracts a square to a segment, and both of these are contractible, the obtained space has the same homotopy type as $X$.
	
	After subdividing the remaining squares by adding a single diagonal edge, this yields a $2$-dimensional simplicial complex on $6$ vertices, which does not contain all possible edges and is therefore homotopy equivalent to a wedge of spheres by Case 3 of Theorem \ref{6vertex}.
\end{proof}

	\begin{theorem}\label{6vertex}
	Let $K$ be a connected simplicial complex on $n\leq6$ vertices. Then $K$ is either the unique minimal triangulation of $\RR P^2$ or is homotopy equivalent to a wedge of spheres.
\end{theorem}

\begin{proof}
	We argue by cases, depending on the dimension of $K$.
	
	{\bf Case 1:} $\dim K\geq 4$. In this case, $K$ is obtained by taking a $4$-simplex and coning off a subcomplex $A$. Therefore, $K\simeq\Sigma A$. Since $A$ is a disjoint union of wedges of spheres by Proposition \ref{5vertex}, it follows that $K$ is a wedge of spheres as well.
	
	{\bf Case 2:} $\dim K=3$. In this case $K$ contains a $3$-simplex, which we may label $1234$. Without loss of generality, we may assume that the vertices $5$ and $6$ are connected. Hence, we are in the situation of Proposition \ref{ssc-split}, where $A$ is spanned by the vertices $\{1,2,3,4\}$ and $B$ by $\{5,6\}$, so $|K|\simeq\Sigma Y$ for some $Y\subseteq A\times_{\mathrm{CW}}B$ with $\dim Y\leq 2$. By Proposition \ref{prism}, the complex $Y$ is homotopy equivalent to a disjoint union of wedges of spheres, so $|K|\simeq\Sigma Y$ is a wedge of spheres as well.
	
	{\bf Case 3:} $\dim K\leq 2$. By Theorem \ref{wall}, it is sufficient to verify that either
the fundamental group is free or we have the triangulation of $\RR P^2$. If the vertices
$V=\{1,2,3,4,5,6\}$ can be split into two subsets that both span triangles, $|K|$ is a suspension
by Proposition \ref{ssc-split}, so its fundamental group is free. Therefore we can assume $K$ does
not contain a pair of opposing triangles. Without loss of generality, there is at least one triangle,
which we label $123$.
	
	Furthermore, we can assume there are no free edges (such edges can always be collapsed). The vertices $\{4,5,6\}$ do not span a triangle. Furthermore, $45$, $46$ and $56$ are edges that are incident to at least two triangles. Suppose not: then we can remove an edge that is not incident to at least two triangles (since it is either free or maximal) and ensure that the remaining two edges are present (by adding them if necessary). This lands us in the scenario of Proposition \ref{ssc-split}, so the complex is a suspension and therefore its fundamental group is free.
	
	Since the complex has no free edges, there must be triangles of the form $12x$, $13y$, $23z$, $a45$, $b46$, $c56$, $d45$, $e46$, $f56$, where $x,y,z\in\{4,5,6\}$ and $a,b,c,d,e,f\in\{1,2,3\}$. By permuting the vertices $\{4,5,6\}$ if necessary, we can assume without loss of generality that $x = 4$. But then the triangle $356$ (which is opposite to $124$) cannot be present, so we must have $c=1$ and $f=2$. This gives two further non-triangles $234$ and $134$.
	
	Now note that since the edge $24$ cannot be free, at least one of the triangles $245$ or $246$ must be present. Since the vertices $5$ and $6$ are still perfectly symmetric, we are free to transpose them and thereby assume without loss of generality that it is the former. Therefore, we have $a=2$ and $136$ is a non-triangle. Therefore $135$ is a triangle and $y=5$, so $246$ is a non-triangle. This leads to $b=1$ and $e=3$ and non-triangles $235$ and $125$. Hence, we must have $z=6$ and a non-triangle $145$. This finally implies that $d=3$ and $126$ is a non-triangle. We obtain the pure $2$-dimensional simplicial complex
	\[
	123,124,156,256,245,135,146,346,236,345,
	\]
	which is exactly the minimal triangulation of $\RR P^2$, see Figure~\ref{rp2}.
\end{proof}

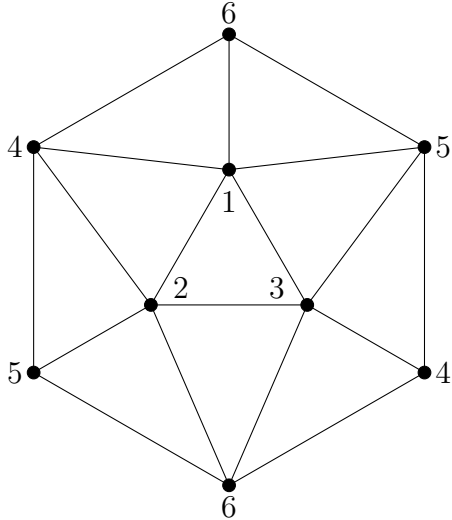
\begin{figure}[h]
	\centering
	
	\begin{tikzpicture}[scale=3]

		\foreach \i [evaluate=\i as \angle using -30+60*(\i-1)] in {1,...,6} {
			\coordinate (A\i) at ({cos(\angle)}, {sin(\angle)});
		}

		\foreach \i [evaluate=\i as \angle using 90+120*(\i-1)] in {1,2,3} {
			\coordinate (T\i) at ({0.4*cos(\angle)}, {0.4*sin(\angle)});
		}

		\draw (T1) -- (T2) -- (T3) -- cycle;
		\draw (A1) -- (A2) -- (A3) -- (A4) -- (A5) -- (A6) --cycle;
		
		\draw (A1) -- (T3) -- (A2) -- (T1) -- (A4) -- (T2) -- (A6) -- (T3);
		
		\draw (T1) -- (A3);
		\draw (T2) -- (A5);

		\filldraw (T1) circle (0.8pt) node[below, inner sep=8pt] {$1$};
		\filldraw (T2) circle (0.8pt);
		\draw (T2)+(0.05,-0.02) node[above right, inner sep=4pt] {$2$};
		\filldraw (T3) circle (0.8pt);
		\draw (T3)+(-0.05,-0.02) node[above left, inner sep=4pt] {$3$};
		
		\filldraw (A1) circle (0.8pt) node[right] {$4$};
		\filldraw (A2) circle (0.8pt) node[right] {$5$};
		\filldraw (A3) circle (0.8pt) node[above] {$6$};
		\filldraw (A4) circle (0.8pt) node[left] {$4$};
		\filldraw (A5) circle (0.8pt) node[left] {$5$};
		\filldraw (A6) circle (0.8pt) node[below] {$6$};

	\end{tikzpicture}

	\caption{The minimal triangulation of a real projective plane $\RR P^2$.}\label{rp2}
\end{figure}

\begin{corollary}\label{corollary:n=6}
Suppose $K$ is a connected simplicial complex on at most $6$ vertices. Then:
\begin{itemize}
\item $\pi_1(K)\cong F_n$, the free group on $n\in\{0,1,\ldots,10\}$ generators, or
\item $\pi_1(K)\cong\ZZ_2$ and $K$ is the minimal triangulation of $\RR P^2$.
\end{itemize}
\end{corollary}

\subsection{Simplicial complex on at most $7$ vertices}\label{subsec:scx up to 7}

By Table \ref{table:numbers_of_complexes} there are $\sim\!4.9\times 10^8$
simplicial complexes on $7$ vertices, up to isomorphism. As we are only interested in their fundamental groups, we may restrict
our attention to connected complexes. Moreover, we have the following easy observations.

\begin{lemma}\label{lemma:reduce_maximal_edge}
	Let $K$ be a connected simplicial complex on $n$ vertices with a maximal edge $(a,b)$. Then either $K' = K \setminus \{(a,b)\}$ is connected and $\pi_1(K) =\pi_1(K') * \ZZ$ or $\pi_1(K) = \pi_1(K' /  a \sim b)$ is the fundamental group of a complex on $n-1$ vertices which is the wedge of two complexes.
\end{lemma}

\begin{proof}
	In the latter case $K'$ has two connected components and we can glue these two parts identifying the vertices $a$ and $b$ forming a connected complex $K' / a \sim b$ on $n-1$ vertices of the same homotopy type as $K$, which is the wedge of the two components of $K'$.
\end{proof}	

\begin{lemma}\label{lemma:2-pure}
If $G$ is the fundamental group of a connected simplicial complex $K$ on $n \geq 3$ vertices, then
there exists a connected $2$-pure complex $K'$ on $n$ vertices such that $G \cong \pi_1(K') * F_r$ for
some $r\geq 0$.
\end{lemma}
\begin{proof}
If $K$ is $1$-dimensional, then $G$ is free and one can simply take  $K'$ to be a $2$-simplex.
Otherwise we may remove from $K$ all higher dimensional simplexes and all $2$-simplices with free edges,
as this doesn't affect its fundamental group.  If the resulting $2$-complex is not $2$-pure, then it has
a maximal edge, and so by Lemma \ref{lemma:reduce_maximal_edge} we can remove this edge and change
its fundamental group at most by a free factor. We may repeat this step until we get a $2$-pure
complex $K'$ satisfying the conclusion of the lemma.
\end{proof}

By the above it is sufficient to determine the fundamental groups of $2$-pure simplicial complexes on
$7$ vertices. There are $h_3(7) = 7013320$ isomorphism classes of such complexes, but some of them are
not connected and some already appear on $6$ vertices so we need to consider $7011181$ connected $2$-pure complexes on $7$ vertices. We used
SageMath to generate all relevant complexes and compute their fundamental groups. More specifically, we
used the \textit{hypergraphs.nauty} function \cite{nauty} to produce all such complexes, and for each of them we
formed \textit{SimplicialComplex} and computed its \textit{fundamental\_group()} using SageMath's
built-in functions. Finally, we further examined how many missing edges they have, obtaining a full
classification of fundamental groups for $7$ vertices.

\begin{theorem}\label{thm:n=7}
The following is the complete list of groups that arise as
fundamental groups of simplicial complexes on at most $7$ vertices.
\begin{itemize}
\item $F_n$, the free group on $n\in\{0,1,\ldots,15\}$ generators,
\item $\ZZ_{2}\ast F_n$, where $n\in\{0,1,\ldots,5\}$, or
\item $\ZZ\times\ZZ$ for the unique minimal triangulation of the torus $T=S^1\times S^1$.
\end{itemize}
\end{theorem}

	\begin{table}[h]
		\begin{tabular}{c|c|c|c|c|c|c|c|c|c|c|c|c}
		
			$1$                           & $\ZZ$   & $F_2$  & $F_3$ & $F_4$ & $F_5$ & $F_6$ & $F_7$ & $F_8$ & $\ZZ_2$ & $\ZZ_2 * \ZZ$ & $\ZZ_2 * F_2$ & $\ZZ^2$ \\ \hline
			5413611 & 1264654 & 276999 & 48944 & 6094  & 463   & 29    & 1     & 1     & 350     & 31            & 3             & 1                            \\ 
		\end{tabular}
		\label{table:groups_for_2-pure_7_vertex}
		\vskip 10pt
		\caption{Distribution of fundamental groups among all $7011181$ isomorphism classes of connected $2$-pure simplicial complexes on $7$ vertices.}
	\end{table}

In fact, we can prove directly that only cyclic groups, up to free factor, and $\ZZ \times \ZZ$ arise as fundamental groups of simplicial complexes on at most $7$ vertices.
Note that the trivial group is cyclic of order $1$.
Therefore, our computer calculations can be used only to verify that group $\ZZ_2$ is the only finite cyclic group that occurs.

\begin{theorem}
Let $K$ be a connected simplicial complex on $7$ vertices. If $\pi_1(K)$ is not cyclic up to a free
factor, i.e. is not of the form $(\ZZ/H) \ast F_n$ for some $H \unlhd \ZZ$, then $K$ is
$2$-dimensional and is the unique minimal triangulation of $T^2$. In particular,
$\pi_1(K) = \ZZ \times \ZZ$.
\end{theorem}
\begin{proof}
Since the fundamental group depends only on the $2$-skeleton, assume for the moment that $K$ is
$2$-pure with no free edges (cf. Lemma \ref{lemma:2-pure_no_free_faces}). Thus each edge has at least
two incident triangles. Suppose $\{1,2,3,4,5,6,7\}$ are vertices of $K$ and, without loss of
generality, consider two triangles $124$ and $234$ with the common edge $24$. Let $A$ be the subcomplex
spanned by $\{1,2,3,4\}$ and let $B$ be the subcomplex spanned by $\{5,6,7\}$. If $B$ is not connected,
we can add edges so that $B$ becomes a tree, and this does not change the property that $\pi_1(K)$ is
not cyclic up to a free factor. But if $B$ is contractible, Proposition \ref{ssc-cone} implies that
$\pi_1(K)$ is, up to a free factor, the quotient of $\pi_1(A)$ that is cyclic, a contradiction.
	
Thus $B$ must be connected and non-contractible, so $B = \partial \Delta^2$ and $567$ is a non-triangle
of $K$. Also because of Proposition \ref{ssc-cone}, $A$ is not contractible. Thus $A$ must contain the
edge $13$ and cannot consist of exactly three triangles. Moreover, if $A$ consists of all four possible
triangles on four vertices (it is a tetrahedron), we can fill it with a $3$-simplex without changing
$\pi_1(K)$, and then again $A$ is contractible.
	
In summary, the pair of triangles $124$ and $234$ with the common edge $24$ forces the existence of
the edge $13$, the edges $56$, $67$ and $57$, and forbids the occurrence of triangles $567$, $123$
and $134$.
	
Let us show that $K$ is a \emph{$2$-incidence complex}, i.e., a $2$-pure complex in which every
edge is a face of exactly two triangles. Suppose to the contrary, that there are
triangles $124$, $234$ and $245$ with the common edge $24$. This leads to $9$ non-triangles.
Moreover, $K$ contains the edges $13$, $15$ and $35$, and each of them is incident to at least
two triangles:
\begin{itemize}
	\item $13a$ and $13b$, where $a,b \in \{5,6,7\}$, since $123$ and $134$ are non-triangles,
	\item $15c$ and $15d$, where $c,d \in \{3,6,7\}$, since $125$ and $145$ are non-triangles,
	\item $35e$ and $35f$, where $e,f \in \{1,6,7\}$, since $235$ and $345$ are non-triangles.
\end{itemize}
Suppose $a=5$ and, because of symmetry, we can assume that $b=6$. Then $156$ is a non-triangle and so
$c,d \neq 6$, hence $c=3$ and $d=7$. But then both $356$ and $357$ are non-triangles, so
$e,f \in \{1\}$, which contradicts the fact that $35$ is not free. Therefore $a,b \neq 5$, and, because
of symmetry, $c,d \neq 3$ and $e,f \neq 1$. Similarly, since $a=6$ and $b=7$, the triangle $156$ is
forbidden, so $c,d \in \{7\}$, a contradiction.
	
Now that we know that $K$ is a $2$-incidence complex we could easily show that its fundamental group
is a surface group up to a free factor. However, we will continue our geometric reasoning to  conclude
that $K$ is in fact the unique minimal triangulation of the torus.
	
Let us start with triangles $124$ and $234$ with the common edge $24$. Without loss of generality,
assume that the edge $13$ is incident to triangles $135$ and $136$. Next, because $156$, $356$ and
$567$ are non-triangles, the edge $56$ is incident to $256$ and $456$. Note that these three pairs of
triangles force the existence of all $\binom{7}{2} = 21$ possible edges, so $K$ has a complete
$1$-skeleton.

The edge $17$ is incident to $1a7$ and $1b7$, where $a,b \in \{2,4,5,6\}$, because $137$ is forbidden.
By symmetry, we can assume that $a=2$. Then the triangles $124$ and $127$ meet at $12$ and so $147$ is
a non-triangle. Therefore, without loss of generality, take $b=5$. This leads to another pair $135$ and
$157$ of incident triangles along the edge $15$. Thus the edges $15$ and $17$ are already incident with
two triangles, so $145$ and $167$ are non-triangles because $K$ is a $2$-incidence complex. So far, we
have the following $8$ triangles in $K$ (in the order in which they are considered):
	\[
	124, 234, 135, 136, 256, 456, 127, 157
	\]
	and $16$ non-triangles (in lexicographic order):
	\[
	123, 125, 134, 137, 145, 147, 156, 167, 245, 246, 247, 257, 346, 356, 357, 567.
	\]
	
	Next, consider the edge $14$ with incident triangle $124$. The second incident triangle must be $146$ as the others are forbidden. Therefore, $126$ is a non-triangle. Now, the vertex $1$ is joined by an edge with all other vertices and each of these edges already has two incident triangles, of which there are $6$ in total. The neighbourhood of $1$ in $K$ is homeomorphic to a disc, and, in fact, one could argue that this is true for every vertex. Since $K$ is a $2$-incidence complex, any non-vertex point of $K$ has also such a neighbourhood, so $K$ is a triangulation of a surface.
	We will not use this fact, and instead we will complete the proof by obtaining an explicit triangulation of the torus:
	\begin{itemize}
		\item The edge $57$ is incident to $157$, and the second triangle must be $457$. Thus $236$ is a non-triangle.
		\item The edge $36$ is incident to $136$, and the second triangle must be $367$.
		\item The edge $26$ is incident to $256$, and the second triangle must be $267$.
		\item The edge $67$ is incident to the triangles $267$ and $367$, so $237$ is a non-triangle.
		\item The edge $27$ is incident to the triangles $267$ and $127$, so $345$ is a non-triangle.
		\item The edge $37$ is incident to $367$, and the second triangle must be $347$. Thus $467$ is a non-triangle.
		\item The edge $35$ is incident to $135$, and the second triangle must be $235$.
	\end{itemize}
	
	In the end, we obtain $14$ triangles in $K$ and $21$ non-triangles, so we have considered all $\binom{7}{3} = 35$ possible triangles. Thus $K$ has the following triangulation:
	\[
	124, 127, 135, 136, 146, 157, 234,
	235, 256, 267, 347, 367, 456, 457,
	\]
	which is exactly the minimal triangulation of $T^2$, see Figure~\ref{csaszar_torus}.
	
	Note that, in this complex all edges are present, no tetrahedra can be added without first adding additional triangles and no edges can be removed without removing some triangles. We can conclude that this triangulation of the torus is the unique complex on $7$ vertices whose group is not cyclic up to a free factor, i.e., our initial assumption of $2$-purity is not needed for the conclusion to hold.
\end{proof}	

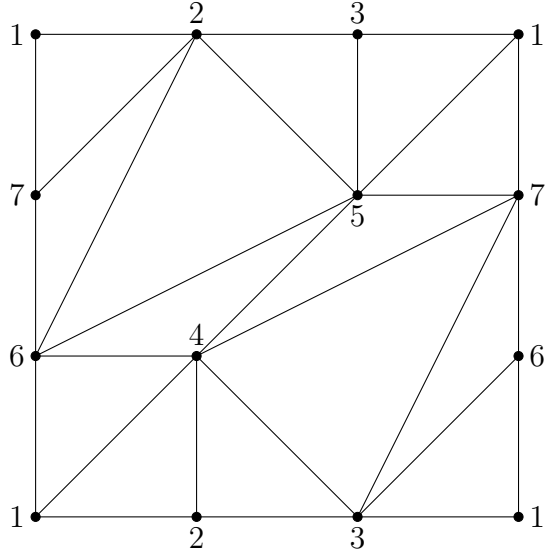
\begin{figure}[h]
	\centering
	
	\begin{tikzpicture}[scale=2.14]
		
		\filldraw (0,0) circle (0.8pt) node[left] {$1$};
		\filldraw (3,0) circle (0.8pt) node[right] {$1$};
		\filldraw (0,3) circle (0.8pt) node[left] {$1$};
		\filldraw (3,3) circle (0.8pt) node[right] {$1$};
		\filldraw (1,0) circle (0.8pt) node[below] {$2$};
		\filldraw (1,3) circle (0.8pt) node[above] {$2$};
		\filldraw (2,0) circle (0.8pt) node[below] {$3$};
		\filldraw (2,3) circle (0.8pt) node[above] {$3$};
		\filldraw (0,1) circle (0.8pt) node[left] {$6$};
		\filldraw (3,1) circle (0.8pt) node[right] {$6$};
		\filldraw (0,2) circle (0.8pt) node[left] {$7$};
		\filldraw (3,2) circle (0.8pt) node[right] {$7$};
		
		\filldraw (1,1) circle (0.8pt) node[above] {$4$};
		\filldraw (2,2) circle (0.8pt) node[below] {$5$};
		
		\draw (0,0) -- (0,1);
		\draw (0,0) -- (1,0);
		\draw (0,0) -- (1,1);
		
		\draw (2,0) -- (1,0);
		\draw (1,1) -- (1,0);
		
		\draw (0,2) -- (0,1);
		\draw (1,1) -- (0,1);
		\draw (2,2) -- (0,1);
		\draw (1,3) -- (0,1);
		
		\draw (2,0) -- (1,1);
		\draw (2,0) -- (3,0);
		\draw (2,0) -- (3,1);
		\draw (2,0) -- (3,2);
		
		\draw (3,2) -- (3,1);
		\draw (3,0) -- (3,1);
		
		\draw (2,2) -- (1,1);
		\draw (2,2) -- (3,2);
		\draw (2,2) -- (3,3);
		\draw (2,2) -- (2,3);
		\draw (2,2) -- (1,3);
		
		\draw (3,3) -- (3,2);
		\draw (1,1) -- (3,2);
		
		\draw (1,3) -- (0,3);
		\draw (1,3) -- (0,2);
		\draw (1,3) -- (2,3);
		
		\draw (0,2) -- (0,3);
		
		\draw (3,3) -- (2,3);

	\end{tikzpicture}
	\caption{The minimal triangulation of a torus --- Cs{\'a}sz{\'a}r torus.}\label{csaszar_torus}
\end{figure}

\section{Complexes on $8$ vertices}
\label{sec:8_ver}

By Table \ref{table:numbers_of_complexes} there are, up to isomorphism, approximately  $\!1.4\times10^{18}$ simplicial complexes on $8$ vertices,  so  we must reduce the number of relevant cases by several orders of magnitude before we feed the data to a computer. By Lemma \ref{lemma:2-pure} it is sufficient to determine the fundamental groups of pure $2$-dimensional simplicial complexes on $8$ vertices. Unfortunately, the number of these complexes is still huge, as $h_3(8) \sim\!1.8\times 10^{12}$, making a direct computation unfeasible.

\subsection{General structural properties.}\label{subsec:General structural properties}
We will achieve a reduction in several steps. To begin, we can assume that the complex has no free faces as these can be collapsed without changing the homotopy type.

\begin{lemma}\label{lemma:2-pure_no_free_faces}
If a connected $2$-pure complex $K$ on $n\geq 4$ vertices has a free face, then there exists a connected $2$-pure complex $K'$ on at most $n$ vertices and without free faces, such that $\pi_1(K) = \pi_1(K') * F_r$ for some $r\geq 0$.
\end{lemma}

\begin{proof}
We perform simplicial collapses until there are no more free faces left. If the complex collapses onto a graph (possibly containing a single point), its fundamental group is free, so we may take $K'=\partial\Delta^3$. Otherwise, we are left with a connected complex $K''$ of dimension $2$, whose maximal faces are either edges or triangles. If removing a maximal edge does not disconnect the complex, we remove it: homotopically, this corresponds to removing an $S^1$ wedge summand, and therefore reducing the rank of the free part in the Grushko decomposition by $1$. We repeat this until the only edges remaining are the ones that disconnect the complex: each such edge may be contracted to a point, i.e. we can remove it and identify its endpoints. At the end of this process we are left with a complex $K'$ with the required properties.
\end{proof}

Choose a vertex $v_0$ of $K$ and divide the simplices of $K$
into those that contain $v_0$ and those that do not. Every simplex containing $v_0$ is a cone over a simplex spanned by the remaining  vertices, so we obtain the following result:
\begin{lemma}
Every simplicial complex $K$ on $n$ vertices can be decomposed as
$K=L\cup CA$, where $L$ is a complex on $n-1$ vertices, $A$
is a subcomplex of $L$ and $CA$ is the cone over $A$.
\end{lemma}

Furthermore, if $K$ is $2$-pure with no free edges, then $L$ is also $2$-pure, otherwise the cone over a maximal edge would give a triangle in which this edge is free. Moreover, if $K$ is connected, we may assume that $L$ is connected due to the following argument. If $L$ is not connected, then $K$ is a wedge sum of two smaller complexes $K_1$ and $K_2$ with the only common vertex $v_0$. However, considering $K_1$ and $K_2$ abstractly, as disjoint complexes, and gluing them along an edge or a $2$-simplex yields a complex $K'$ on $n-1$ vertices, such that $\pi_1(K) \cong \pi_1(K')$. Thus, in this case no new fundamental group occurs. Additionally, if $L$ is $2$-pure on $7$ vertices and is not connected, then its components have free fundamental groups.

In addition to $L$ being connected, if we only care about fundamental groups up to a free factor, we can also assume that $A$ is connected: if $A$ has $r$ components $A_1,\ldots, A_r$, we can find a path consisting of edges in $L$ that connects two of these components. Adding these edges to $A$ has the effect of reducing the rank of the free group factor in $\pi_1(K)$ by $1$. We can repeat this procedure until $A$ becomes connected.

Proposition \ref{split-suspension_and_cones} easily applies to the splitting of a complex into a vertex and a connected subcomplex on the remaining vertices:

\begin{corollary}\label{corollary:quotient_when_coning}
Suppose $L$ is a connected simplicial complex and $A$ is a subcomplex of $L$ that has $r$ connected components.  Then there is a normal subgroup $H\unlhd \pi_1(L)$ such that
\[
\pi_1(L\cup CA)\cong\big(\pi_1(L)/H\big)\ast F_{r-1}.
\]
\end{corollary}

\begin{remark}
Observe that $H$ is the normal closure of homotopy classes of loops in $L$ that are contained in $A$. This subgroup can be quite different from $\pi_1(A)$.
\end{remark}

In particular, if $\pi_1(L)$ is trivial, then
$\pi_1(L\cup CA)$ is free, and if $\pi_1(L)$ is cyclic, then
$\pi_1(L\cup CA)$ is also cyclic up to a free group factor in its Grushko decomposition.

\

In our computations we do not need to consider all possible subcomplexes $A\leq L$. If $A$ contains a $2$-simplex, the cone over this $2$-simplex is a $3$-simplex in $K$ that can be removed without
affecting the fundamental group. Therefore it is sufficient to take $A$  at most $1$-dimensional. Furthermore, if $L$ has a free edge, this edge will necessarily have to be included in $A$, as we assume that $K$ has no free edges. Finally, because $L$ and $A$ are connected and $A$ is at most $1$-dimensional, if there is a vertex of $L$ not in $A$, we can find an edge in $L$ that can be added to $A$ without changing $\pi_1(L \cup CA)$. Thus, the resulting complex after coning can have free edges, but these will be incident to the cone vertex $v_0$.
As a consequence we may assume that $A$ contains a spanning tree of $L$.

Summarizing all the above considerations, we obtain the following structural theorem.

\begin{theorem}\label{theorem:structural_theorem}
	Let $G$ be a finitely presented group that does not admit a nontrivial free group as a free factor, and let $n$ be the minimal integer such that  there exists a simplicial complex $K$ on $n$ vertices with $\pi_1(K) \cong G * F_{r'}$ for some $r'\geq 0$. Then there exist
\begin{itemize}
\item[1)] a connected, $2$-pure simplicial complex $L$ on $n-1$ vertices with $\pi_1(L) \neq 1$, and
\item[2)] a connected, $1$-dimensional subcomplex $A\le L$ containing a spanning tree of $L$ and all
free edges of $L$,
\end{itemize}
such that for some $r\geq 0$
	\[	  \pi_1( L \cup CA) \cong G * F_r.	\]
 Moreover, if $G$ is not cyclic, then we can additionally assume that $\pi_1(L)$ is not cyclic.
\end{theorem}

\subsection{Algorithm and implementation}

Based on the structural theorem, we can devise an algorithm that determines all possible fundamental groups of simplicial complexes with at most $n$ vertices.
We will use the following notation:

\begin{itemize}
	\item  $\Pure{n}$ --- the set of representatives of all isomorphism classes of connected $2$-pure simplicial complexes $L$ on $n$ vertices such that $\pi_1(L) \neq 1$,
	\item $G_n$ --- the set of representatives of all isomorphism classes of groups $G$ without a nontrivial free factor, such that $\pi_1(K) \cong G * F_r$ for some $K \in \Pure{n}$ and $r \geq 0$.
\end{itemize}

\RestyleAlgo{ruled}
\begin{algorithm}
	\caption{Determination of all fundamental groups of complexes on $n$ vertices.}\label{alg:one}
	\KwData{$n \geq 4$, $\Pure{n-1}$}
	\KwResult{$G_n$}
	$G_n \gets \varnothing$
	
	\ForAll{ $L \in \Pure{n-1}$}{
	
	$A \gets $ all free edges of $L$
	
	$E \gets [e_1,\ldots,e_k]$ --- the list of all edges of $L$ that are not free
	
	EXTEND($L$,$A$,$E$,$0$)
	}	
	\Return $G_n$
\end{algorithm}

\begin{algorithm}
	\caption{EXTEND($L$,$A$,$E$,$i$) --- Recursive extension of the complex}\label{alg:two}
	\KwData{$L$ --- connected $2$-pure simplicial complex, \\
	$A$ --- subcomplex of $L$, \\
	$E$ --- list of non-free edges of $L$.}
	
	\If{$i > {\rm SIZE}(E)$}{\Return}
	
	EXTEND($L$,$A$,$E$,$i+1$)
	
	$A \gets A \cup E[i]$
	
	\If{$A$ contains a spanning tree of $L$}
	{
	$G \gets \pi_1(L \cup CA)$

	\If{ $G \neq 1$}
	{
		\If{$G \notin G_n$ (up to isomorphism)}{
		$G_n \gets G_n \cup \{G\}$}
		
		EXTEND($L$,$A$,$E$,$i+1$)
	}
	}
	
\end{algorithm}

\ \\
The correctness of Algorithm \ref{alg:one} can be deduced from Theorem \ref{theorem:structural_theorem} as follows:

\begin{lemma}
	If $G$ is a finitely presented group without a free group factor such that $G \ast F_r$ is the fundamental group of a simplicial complex on $n$ vertices, then the above algorithm finds a group $G \ast F_{r'}$ for some $r'$.
\end{lemma}

\begin{proof}
	During the execution of the algorithm, we consider each extension of the connected $2$-pure complex $L$ on $n-1$ by coning over a subcomplex $A$ of $L$ that contains a spanning tree of $L$ and all its free edges. By Theorem \ref{theorem:structural_theorem} we are sure to get a complex with the desired fundamental group.
\end{proof}

\begin{remark}
	Of course, there are many possible ways of searching all subcomplexes $A$. The advantage of
	considering them in the inclusion-preserving manner is the fact that we stop further extensions if
	$\pi_1(Y \cup CA) =1$, since then for any $A'$ containing $A$, $\pi_1(Y \cup CA')$ is a quotient of
	$\pi_1(Y \cup CA)$, so it is also trivial. This remark allows a significant reduction of
the number of cases that need to be computed.
\end{remark}

Since the number of cases that we need to check is still very large, we wrote a program in C++ \cite{c++git} that
extremely accelerated calculations compared to SageMath. However, we still need to
reduce the presentation of the group obtained from the simplicial complex and to recognise it. Our
program is not a universal tool for the recognition of fundamental groups of complexes in C++ at
the moment. It is adapted to the specific groups obtained from complexes on $8$ or $9$ vertices in the
course of successive calculations. However, it has the potential for further development.

The algorithm described above has been implemented faithfully in C++ and differs only in technical issues and the need to execute particular steps.

\subsection{Computations and results for $n=8$}

From the $7011181$ connected $2$-pure complexes on $7$ vertices only $1597570$ of them have a nontrivial fundamental group. The number of extensions of any such complex is bounded above by the number $2^{\binom{7}{2}} = 2 097 152$ of all subsets of its edges, so there are at most
\[
2^{\binom{7}{2}} \cdot 1597570  \sim\,3.35 \times 10^{12}
\]
fundamental groups to compute in total, while $h_3(8) \sim\!1.8\times 10^{12}$ which shows
that this is a very coarse estimate.

However, it turned out that in our implementation it was sufficient to calculate
$\sim\!1.45 \times 10^{10}$ groups. This is a more reasonable number, but still too big to simply
determine these groups in SageMath. Calculations made in C++ using a custom-made structure for finitely
presentable groups \cite{c++git} turned out to be sufficient. They took $562$ CPU hours which lasted about $2.5$ days
on $9$ threads using Intel Core i5-1334U. The result is easily reproducible. Among the
group presentations obtained in this way only $258$ could not be recognised by our structure in C++, so they were saved
separately
and checked in SageMath.

The following theorem provides a complete list of groups that arise
as fundamental groups of simplicial complexes on at most $8$ vertices.
The explicit triangulations realizing these groups can be found in \cite{dataset}.

\begin{theorem}\label{thm:n=8}
	
If $K$ is a connected simplicial complex on at most $8$ vertices, then
there exists an integer $r$, $0\leq r\leq 21$, such that
\[\pi_1(K)\cong G\ast F_r,\]
where $G$ is one of the following groups:
\begin{itemize}
	\item $\ZZ_2$, $\ZZ_3$, $\ZZ_4$,
	\item $\ZZ\times\ZZ$, $\ZZ\rtimes\ZZ$ (fundamental groups of the torus and the Klein bottle, respectively),
	\item $B_3$ (the braid group on $3$ strands or
	the trefoil knot group),
\end{itemize}
$F_r$ is a free group of rank $r$ and
the range of admissible ranks $r$ depends on $G$ and admits the value $0$ in every case.
	
\end{theorem}

\begin{remark}
	The calculations can be significantly reduced if we choose to focus on non-cyclic groups. By Proposition \ref{corollary:quotient_when_coning} it is sufficient to consider $L \in \Pure{7}$ such that $\pi_1(L)$ is not cyclic. There are $332566$ such complexes. Moreover, in the function EXTEND, instead of $G\neq 1$ one can check if $G$ is not cyclic. Thanks to these restrictions, we only need to consider $1.58 \times 10^8$ cases and groups to calculate. The calculations take about 2 hours and 20 minutes on a single thread running on Intel Core i5-1334 using a C++ program \cite{c++git}.  Let us note that by considering all extensions of such complexes $L$ we get about $4.5\times 10^9$ possible cones, so there is a clear advantage to using our algorithm.
	
	The number of complexes in this approach is small enough the full computation can be done in Mathematica, and indeed, this was our original approach to the problem. The Mathematica implementation is also available online, see \cite{mathematica-git}. Since all isomorphism classes of $3$-uniform hypergraphs on $7$ vertices are considered (rather than only those that do not occur on $6$ vertices), we obtain $332710$ complexes on $7$ vertices that have non-cyclic fundamental groups. As above, these lead to about $4.5\times 10^9$ cones. Each of these is then analyzed by checking all possible splittings into two parts with $4$ vertices each. If one of these parts contains $\geq 2$ triangles and the other part contains $\geq 3$ triangles, with respect to any such splitting, the complex can be discarded, since its group will be cyclic up to a free factor by Proposition \ref{ssc-cone}. These checks are purely combinatorial, so they are relatively fast, but the reduction is significant, as we are left with only $3807843$ complexes to consider. The fundamental groups of these complexes can then be computed directly in SageMath, leading to only $1336$ distinct presentations with nontrivial relations, each of which can be checked to represent one of the groups in Theorem \ref{thm:n=8}. The computation took about $3$ days using a single thread on an Intel Core i7-7600U.
\end{remark}

The theorem immediately yields new values of the Karoubi--Weibel invariant:
\begin{corollary}
	Let $G$ be a group without nontrivial free factors
	in its Grushko decomposition. Then $\KW(G) = 8$ if and only if $G$ is isomorphic to one of the following groups:
	\[
	\ZZ_3,\ \ZZ_4,\ \ZZ\rtimes\ZZ,\ B_3.
	\]
\end{corollary}

\begin{example}
	Let us describe a simplicial complex on $8$ vertices whose fundamental group is isomorphic to $B_3$, the braid group on $3$ strands or equivalently, the trefoil knot group. The list of maximal simplices is
	\begin{align*}	
	[[0, 1, 4], [0, 1, 7], [0, 2, 3], [0, 2, 5], [0, 3, 4], [0, 5, 6], [0, 6, 7], [1, 2, 3], [1, 2, 4], [1, 3, 6],\\
	[1, 5, 6], [1, 5, 7], [2, 4, 7], [2, 5, 7], [3, 4, 5], [3, 5, 7], [3, 6, 7], [4, 5, 6], [4, 6, 7]]
	\end{align*}	
schematically represented in Figure \ref{figure:complex_3_braid_group}.

	\begin{figure}[h]
		\includegraphics[width=295pt]{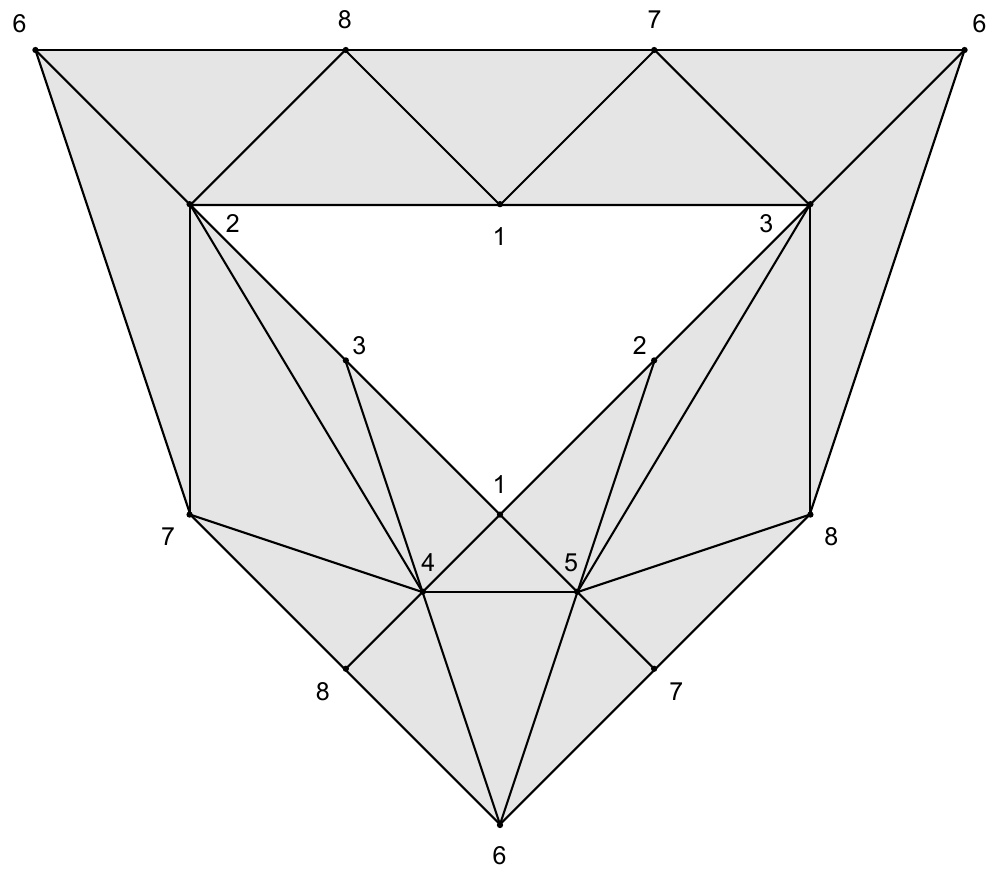}
		\caption{Simplicial complex on $8$ vertices with the fundamental group isomorphic to the braid group on $3$ strands.}
		\label{figure:complex_3_braid_group}
	\end{figure}

\end{example}

\begin{proposition}
	The complex $K$ presented in Figure \ref{figure:complex_3_braid_group} is a deformation retract of the complement of the trefoil knot $3_1$. Therefore the covering type of $\RR^3 \setminus 3_1$ is equal to $8$.
\end{proposition}

\begin{proof}
 The complex $K$ is a triangulation of the space $X_{2,3}$, where $X_{m,n}$ is homeomorphic to the quotient of $S^1 \times [0,1]$ under the identifications $(x,0) \sim (x+1/m,0)$ and $(x,1) \sim (x+1/n,1)$ for $x \in S^1 = [0,1]/\{0,1\}$. The space is described in Example 1.24 of A.~Hatcher's book \cite{Hatcher}, where it is shown that there is a deformation retraction from $\RR^3 \setminus 3_1$ onto $X_{2,3}$.
 By Theorem \ref{thm:n=8} the minimal number of vertices for which there is a complex that realizes the trefoil group is $8$, therefore
  the covering type of $K$ is $8$.
\end{proof}

To conclude, let us consider the implications of Theorem \ref{thm:n=8}
for the Bj\"orner--Lutz conjecture
\cite[Conjecture 6]{Bjorner-Lutz} that we mentioned in the
Introduction. Assume that $K$ is a simplicial complex on $n$ vertices
whose geometric realization is homeomorphic to the Poincar\'e
homology sphere. Since $K$ is $3$-dimensional, there are
four vertices in $K$ spanning a simplex in $K$. Let
$K'$ denote the subcomplex of $K$ spanned by the remaining
$n-4$ vertices. Applying 
Remark \ref{removing-a-ball}, we obtain
$\pi_1(K')=\pi_1(K)$, the binary icosahedral group $\SL(2,5)$.
Since $\SL(2,5)$ does not appear in the list given in Theorem
\ref{thm:n=8}, we conclude that $K'$ has at least $9$ vertices.
Thus we have:
\begin{theorem}\label{thm:Poincare}
	The number of vertices in any triangulation of the Poincar\'e sphere $P$, or any other $3$-manifold that is a nontrivial homology sphere, is at least $13$.
\end{theorem}

The result improves the previous estimate  of Bagchi and Datta \cite{Bagchi-Datta} who proved that the minimal number of vertices
required to triangulate the Poincar\'e sphere is at least $12$.

We may also observe that the above estimate can be improved if
the minimal triangulation $K$ admits a nontrivial bistellar move
(see \cite{Bagchi-Datta} and \cite{Bjorner-Lutz}). In that case one can find two
$3$-dimensional simplices in $K$ that intersect in a $2$-dimensional
face and such that the opposite vertices do not form an edge in $K$.
Let $K'$ denote the subcomplex of $K$ spanned by the remaining $n-5$ vertices. Then we may again apply 
Remark \ref{removing-a-ball}
to show that $\pi_1(K')=\pi_1(K)$. Since $K'$ has at least $9$
vertices, we may conclude that $K$ must contain at least $14$ vertices.

\section{Further applications}

\label{sec:applications}

\subsection{Complexes on $9$ vertices}

The numbers $r_9 \approx 7.89\times 10^{35}$, of isomorphism classes of simplicial complexes, and $h_3(9) \approx 5.33\times 10^{19}$, of isomorphism classes of $2$-pure simplicial complexes on $9$ vertices, are drastically larger than the corresponding numbers for $8$ vertices. Moreover, our geometric reductions summarized in Theorem \ref{theorem:structural_theorem}, as well as the algorithm based on them, require the use of the set $\Pure{n}$ of (isomorphism classes of) connected $2$-pure complexes on $n$ vertices with nontrivial fundamental groups. Since the objective of this work was to determine the fundamental groups of complexes on $8$ vertices without determining the set $\Pure{8}$, we cannot use this set to obtain a classification of groups for $9$ vertices. Furthermore, even if we had this set, our preliminary computations indicate that its size would be of order at least $10^9$. Together with the number of extensions of each such complex by attaching a cone along a subcomplex, which is bounded above by $2^{\binom{8}{2}} \approx 2.7 \cdot 10^8$, the number of complexes would still be too large for calculation.

One natural approach we could try instead is to perform double coning on a complex with $7$ vertices. This eliminates the issue of computing the set $\Pure{8}$, but the number of possible extensions of these complexes by two vertices is again too large for exhaustive computation. Furthermore, it is not obvious whether we can assume that the $7$-vertex complex is $2$-pure. This is certainly an area for further investigations.

Due to the greater complexity and multitude of presentations of fundamental groups of complexes on $9$ vertices, as indicated by our preliminary calculations (see examples below), we believe that their complete classification may be very difficult to achieve. Apart from the need for additional geometric reductions, it may be very important to focus on a class of groups of a specific type, e.g. only perfect groups, to achieve at least another step towards the Bj\"orner--Lutz conjecture.

Since we have given a complete classification of groups for $8$ vertices, every other group that occurs as the fundamental group of a complex on $9$ vertices must automatically have the Karoubi--Weibel invariant equal to $9$.

Below we provide the results of our experiments using our approach or obtained as extensions of $2$-incidence complexes on $9$ vertices. The particular triangulations can be found in \cite{dataset}.

\begin{theorem} \label{result:9_vertices}
	The following groups $G$ have $\KW(G) = 9$:
	\begin{itemize}
		\item $\ZZ_m$ for $m\in \{5,6,7,8,9\}$,
		\item $\ZZ_2 \ast \ZZ_2 = D_\infty$ (infinite dihedral group),
		\item $\ZZ_2 \ast \ZZ_3$,
		\item $\ZZ \times \ZZ_m$ for $m\in \{2,3\}$,
		\item $D_6$ (dihedral group of order $6$),
		\item $Q_8$ (quaternion group of order $8$),
		\item $BS(2,1)$, $BS(2,2)$, $BS(2,-2)$, $BS(3,1)$, $BS(3,-1)$ (Baumslag--Solitar groups),
		\item $\pi_1(M_2)$ (the fundamental group of a closed orientable surface of genus $2$),
		\item $\pi_1(N_g)$ for $g\in \{3,4,5\}$ (the fundamental group of a closed non-orientable surface of genus $g$),
		\item $\pi_1(X_{2,4}) = \langle\, a,\, b\, |\, a^4 = b^2\, \rangle$ (see \cite[Example 1.24]{Hatcher}).
	\end{itemize}
\end{theorem}

\begin{example}
	The dihedral group $D_6$ of order $6$ and the quaternion group $Q_8$ of order $8$ are the first (and for now the only known) finite non-cyclic groups which occur as the fundamental groups of simplicial complexes on $9$ vertices. Triangulations that yield $D_6$ are not unique, here is an example with $34$ edges and $26$ triangles:
	\begin{align*}
		[[0, 1, 4], [0, 1, 7], [0, 2, 3], [0, 2, 5], [0, 3, 4], [0, 5, 6], [0, 6, 7], [1, 2, 3], [1, 2, 4], [1, 3, 6],\\
		[1, 4, 8], [1, 5, 6], [1, 5, 7], [1, 6, 8], [2, 3, 8], [2, 4, 6], [2, 4, 7], [2, 5, 7], [2, 6, 8], [3, 4, 5], \\
		[3, 5, 7], [3, 6, 7], [3, 7, 8], [4, 5, 6], [4, 6, 7], [4, 7, 8]].
	\end{align*}
	This complex is an extension of the above-presented complex whose fundamental group is~$B_3$ (Figure \ref{figure:complex_3_braid_group}).
	
	A complex realizing $Q_8$ is given as follows:
	\begin{align*}
		[[0, 1, 2], [0, 1, 3], [0, 1, 6], [0, 2, 3], [0, 2, 7], [0, 4, 5], [0, 4, 6], [0, 4, 7], [0, 5, 6], [0, 5, 8],\\
		[0, 7, 8], [1, 2, 3], [1, 3, 4], [1, 4, 5], [1, 4, 8], [1, 5, 7], [1, 6, 7], [1, 6, 8], [2, 3, 5], [2, 4, 5],\\
		[2, 4, 8], [2, 5, 6], [2, 6, 7], [2, 7, 8], [3, 4, 6], [3, 4, 7], [3, 5, 7], [3, 5, 8], [3, 6, 8]].
	\end{align*}
\end{example}

\subsection{Aspherical spaces}

We would like to point out the relationship of our considerations to aspherical spaces.
 Often invariants of a group $G$ (e.g. its topological complexity) are defined as homotopy invariants of its Eilenberg--MacLane space $K(G,1)$ which is unique up to homotopy type. Thus one can define the covering type $\ct(G)$ of $G$ to be $\ct(K(G,1))$. If this number is finite, $G$ is called geometrically finite. For example, $\ZZ_2$ is not geometrically finite, since we have:
 \[
 \KW(\ZZ_2)=\ct(\RR P^2)=6,\qquad\text{while}\qquad\ct(K(\ZZ_2,1))=\ct(\RR P^{\infty})=\infty.
 \]
 In general, $\KW(G) \leq \ct(G)$. However, when $G$ is geometrically finite, one might expect these numbers to be more closely related, so the question of precise relationship between them becomes important. In particular, when does the equality $\KW(G) = \ct(K(G,1))$ hold? It is not clear even for surface groups. Recall that E. Borghini and E.\,G. Minian \cite{Borghini-Minian} showed that the covering type of a closed surface $S$ is equal to the number of vertices in its minimal triangulation with only one exception, the orientable surface $M_2$ of genus $2$, for which $\ct(M_2)$ is one less. In general, the existence of a~non-aspherical complex $K$ such that $\pi_1(K) \cong \pi_1(S)$ and with fewer than $\ct(S)$ vertices cannot simply be ruled out without proof. Our computations up to $9$ vertices confirm that $\KW(\pi_1(S)) = \ct(\pi_1(S))$ if $0 \geq \chi(S) \geq -3$, thus covering the three exceptional cases in \cite{Borghini-Minian}.

However, $\ct(K(G,1))$ can be much bigger than $\KW(G)$ for higher-dimensional spaces $K(G,1)$ (for groups of higher cohomological dimension). For example, consider $G = \ZZ^n$ for which $K(G,1) = T^n$ is an $n$-dimensional torus. Since the Lusternik--Schnirelmann category of $T^n$ is ${\rm cat}(T^n) = n+1$, by \cite[Theorem 2.2]{GMP} we get
\[
\ct(T^n) \geq \frac{1}{2}{\rm cat}(T^n) ( {\rm cat}(T^n)+1) = \frac{(n+1)(n+2)}{2},
\]
and so it is at least quadratic in $n$. On the other hand, an explicit construction provided by F. Frick and M. Superdock \cite{Frick-Superdock} for $n\geq 7$ yields that
\[
\KW(\ZZ^n)\leq 4n+(-1)^{n+1},
\]
so it is at most linear in $n$. We were able to construct explicit complexes (available at \cite{dataset}) showing that
\[
\KW(\ZZ^3)\leq 11,\quad\KW(\ZZ^4)\leq 13,\quad\KW(\ZZ^5)\leq 19\quad\text{and}\quad\KW(\ZZ^6)\leq23.
\]
This shows that the Frick--Superdock upper bound is in fact valid for all $n\in\NN$. 
Furthermore, combined with the above quadratic bound for $\ct(T^n)$, we can therefore conclude that
\[
\KW(\ZZ^n) < \ct(K(\ZZ^n,1))
\]
for all $n\geq 4$. For $n=3$ by Theorem \ref{thm:n=8}, the bound in \cite[Example 3.9]{GMP} and classical triangulation of $T^3$ with $15$ vertices \cite{Kuehnel-Lassmann} we get 
\[
9 \leq \KW(\ZZ^3) \leq 11 \leq \ct(K(\ZZ^3,1)) \leq 15.
\]

\subsection{PL triangulations of manifolds}

Recall that a homology sphere is a manifold that has the same integral homology as a sphere.
Our computations imply that every nontrivial $3$-dimensional homology sphere needs at least $13$ vertices to triangulate. This can be used to improve criteria from \cite{Pavesic} to recognize whether a triangulation of a manifold is combinatorial, thus giving it the structure of a PL manifold. Recall that a triangulation is combinatorial if the link of every 
vertex is PL homeomorphic to a sphere.

In fact, our considerations for the Poincar\'e sphere apply to combinatorial triangulations of homology spheres of any dimension $d\geq 3$. By the Hurewicz theorem the fundamental group $\pi_1(S)$ of such a homology sphere $S$ is a perfect group. If $S$ has a combinatorial triangulation with $n$ vertices, we can apply Remark \ref{removing-a-ball} to remove a $d$-simplex and obtain a simplicial complex $L$ on $n-d-1$ vertices such that $\pi_1(L)\cong\pi_1(S)$. Since there are no perfect groups on $8$ vertices, we get the following:
\begin{corollary}\label{corollary: triangulations of homology spheres}
	A combinatorial triangulation of a nontrivial homology $d$-sphere requires at least $d+10$ vertices.
\end{corollary}
Hence our result actually improves, by $1$, Bagchi and Datta's result \cite[Corollary 5]{Bagchi-Datta} for combinatorial $\ZZ$-homology spheres (see also \cite[Corollary 3]{Bagchi-Datta} for the case where $S$ is a $\ZZ_2$-homology sphere). Furthermore, our method is quite different.

This can be used to recognize combinatorial triangulations of manifolds. Namely, if all links of vertices in a simplicial complex are homology spheres, by counting their vertices we can, in some range, deduce whether they are standard spheres. We can also use the fact that for dimension $d \leq 4$ all triangulations of closed $d$-manifolds are combinatorial, see \cite[Proposition 3.11]{Friedl-etAl}. 
This mainly leads to an improvement in dimensions $4$ and $5$:

\begin{corollary}[{cf. \cite[Theorem 1.5]{Pavesic}}]\label{corollary:combinatorial_triangulations}
	Let $K$ be a $4$-dimensional (resp. $5$-dimensional) simplicial complex such that the link of every	vertex is a $3$-dimensional (resp. $4$-dimensional) homology sphere that has at most $12$ (resp. $13$) vertices. Then $K$ is a combinatorial triangulation of a $4$-dimensional (resp. $5$-dimensional) PL manifold.
\end{corollary}

\begin{remark}
	Note that the criterion in \cite[Theorem 1.5]{Pavesic} requires at most $9$ vertices for $3$-dimensional homology spheres and $12$ vertices for $4$-dimensional homology spheres. The same criterion can also be used in simplicial complexes of higher dimension to verify whether links of simplices are combinatorial spheres, but we only get an improvement in codimensions $4$ and $5$.
\end{remark}

\end{document}